\documentclass{amsart}
\usepackage{graphicx}
\usepackage{amsmath}
\usepackage{amssymb}
\usepackage{oldgerm}
\usepackage{mathdots}
\usepackage[all]{xy}

\newcommand{\Hom}{\operatorname{Hom}\nolimits}
\renewcommand{\Im}{\operatorname{Im}\nolimits}

\newcommand{\Ker}{\operatorname{Ker}\nolimits}

\newcommand{\Id}{\operatorname{id}\nolimits}

\newcommand{\Ann}{\operatorname{ann}\nolimits}

\newcommand{\Tor}{\operatorname{Tor}\nolimits}
\newcommand{\Ext}{\operatorname{Ext}\nolimits}

\newcommand{\Ho}{\operatorname{H}\nolimits}

\newcommand{\rank}{\operatorname{rank}\nolimits}
\newcommand{\m}{\operatorname{\mathfrak{m}}\nolimits}

\newcommand{\cx}{\operatorname{cx}\nolimits}
\newcommand{\lcx}{\operatorname{\ell cx}\nolimits}

\newcommand{\comments}[1]{}

\newtheorem{theorem}{Theorem}[section]

\newtheorem{corollary}[theorem]{Corollary}

\newtheorem{lemma}[theorem]{Lemma}
\newtheorem{proposition}[theorem]{Proposition}
\newtheorem*{theoremempty}{Theorem}
\theoremstyle{definition}

\theoremstyle{definition}
\newtheorem*{question}{Question}

\newtheorem{example}[theorem]{Example}
\theoremstyle{definition}

\theoremstyle{remark}
\newtheorem*{remark}{Remark}
\theoremstyle{remark}

\theoremstyle{definition}

\begin{document}

\title{Homological algebra modulo exact zero-divisors}

\author{Petter Andreas Bergh, Olgur Celikbas and David A. Jorgensen}

\address{Petter Andreas Bergh \\ Institutt for matematiske fag \\
NTNU \\ N-7491 Trondheim \\ Norway} \email{bergh@math.ntnu.no}
\address{Olgur Celikbas \\ Department of Mathematics \\ University of Kansas \\ Lawrence \\ KS 66045 \\ USA} 
\curraddr{Department of Mathematics, University of Missouri, Columbia, MO 65211, USA}
\email{celikbaso@missouri.edu}
\address{David A.\ Jorgensen \\ Department of Mathematics \\ University
of Texas at Arlington \\ Arlington \\ TX 76019 \\ USA}
\email{djorgens@uta.edu}



\date{\today}

\begin{abstract}
We study the homological behavior of modules over local rings modulo exact zero-divisors.
We obtain new results which are in some sense ``opposite" to those known for modules over local rings modulo regular elements.
\end{abstract}

\maketitle

\section{Introduction}

Given a local (meaning also commutative and Noetherian) ring $S$ and an ideal $I$, one may ask whether the homological behavior of modules over $S/I$ is related to that over $S$. In general this is hopeless; one needs to restrict the ideal $I$. When $I$ is generated by a regular sequence, there is a well-developed and powerful theory relating the homological properties of modules over these two rings (see for instance \cite{Avramov1}, \cite{AvBu}, \cite{Be}, \cite{Dao}, \cite{Eisenbud}, \cite{Gulliksen1}, \cite{Jo1}, and \cite{Jo2}). For example, suppose that the ideal $I$ is generated by a single regular element $x\in S$, denote the factor ring $S/(x)$ by $R$, and let $M$ and $N$ be $R$-modules. Then a primary result is that
$$\Tor^R_1(M,N) = \Tor^R_2(M,N) = \cdots = \Tor^R_n(M,N) =0$$ for some $n \ge 2$, implies 
$$\Tor^S_2(M,N) = \cdots = \Tor^S_n(M,N) =0.$$ In other words, vanishing of homology over $R$ implies the vanishing of homology over $S$. An analogous statement for cohomology also holds.

Another primary result compares the complexity of a finitely generated $R$-module with its complexity as an $S$-module (see Section $3$ for the definition of complexity). Namely there are inequalities \cite[3.2(3)]{Avramov1}:
$$\cx_S(M) \le \cx_R(M) \le \cx_S(M) +1.$$

In this paper we study the case where the element $x \in S$ is in some sense the ``next best thing" to being a regular element. More precisely, we consider the case where the annihilator of $x$ is a nonzero principal ideal whose annihilator is also principal (and therefore is the ideal $(x)$). Following \cite{HenriquesSega}, the element $x$ is said to be an \emph{exact zero-divisor} if it is nonzero, belongs to the maximal ideal of $S$, and there exists another element $y \in S$ such that $\Ann_S (x) = (y)$ and $\Ann_S (y) = (x)$. In this case we say that $(x,y)$ is a \emph{pair of exact zero-divisors} of $S$. The ideal $(x)$ is then an example of a \emph{quasi-complete intersection ideal}, a notion introduced in \cite{AvramovHenriquesSega}. In that same paper results relating certain invariants of modules over $S$ with those over $S/(x)$ are proved. We continue along similar, but more homological lines, and show that even if the element $x$ is the next best thing to being regular, namely an exact zero-divisor, then the homological relationships between $S/(x)$-modules over $S/(x)$ and over $S$ change dramatically compared to the case where $x$ is regular. Two of our main results in Section 2 concern the vanishing of (co)homology. In particular the result for homology takes the following form:

\begin{theoremempty}
Let $R=S/(x)$ where $S$ is a local ring and $(x,y)$ is a pair of exact zero-divisors in $S$. Furthermore, let $M$ and $N$ be $R$-modules such that $yN=0$. If there exists an integer $n \ge 2$ such that $\Tor^{R}_{i}(M,N)=0$ for $1 \le i \le n$, then $\Tor^{S}_{i}(M,N) \cong M \otimes_S N$ for $1 \le i \le n-1$.
\end{theoremempty}

Compared to the vanishing result for the case where $x$ is regular, the conclusion of the previous theorem is opposite in the sense that it is a non-vanishing result: vanishing of homology over $R$ implies the \emph{non-vanishing} of homology over $S$ (when the modules involved are nonzero and finitely generated$.$) For cohomology, we obtain the following analogue:

\begin{theoremempty}
Let $R=S/(x)$ where $S$ is a local ring and $(x,y)$ is a pair of exact zero-divisors in $S$. Furthermore, let $M$ and $N$ be $R$-modules such that $yN=0$. If there exists an integer $n \ge 2$ such that $\Ext_{R}^{i}(M,N)=0$ for $1 \le i \le n$, then $\Ext_{S}^{i}(M,N) \cong \Hom_S (M,N)$ for $1 \le i \le n-1$.
\end{theoremempty}

We also compare the complexities of finitely generated modules over $R$ and over $S$. Similar to our previous results, we show that such a comparison is quite different from the case where $x$ is regular. The following theorem is the main result of section $3$.

\begin{theoremempty}
Let $R=S/(x)$ where $S$ is a local ring and $x$ is an exact zero-divisor in $S$. If $M$ is a finitely generated $R$-module, then for any $n$ there are inequalities
$$\beta^R_n(M) - \sum_{i=0}^{n-2} \beta^R_i(M) \le \beta^S_n(M) \le \sum_{i=0}^n \beta^R_i(M)$$
of Betti numbers. In particular, the inequality $\cx_S(M) \le \cx_R(M)+1$ holds.
\end{theoremempty}

In the final section, Section 4, we discuss canonical endomorphisms of complexes of finitely generated free $R$-modules, and canonical elements of $\Ext^2_R(M,M)$, for finitely generated $R$-modules $M$, in the case
where $R=S/(x)$ and $(x,x)$ is a pair of exact zero-divisors of $S$.  The main result, Theorem \ref{thmlift}, equates the ability to lift a finitely generated $R$-module $M$ from $R$ to $S$ to the triviality of the canonical element in $\Ext^2_R(M,M)$.  This generalizes classical results (see, for example, \cite{ADS}) on lifting modules from $T/(x)$ to $T/(x^2)$ in the case where $x$ is a non-zero-divisor of the local ring $T$ (cf. Example \ref{noembdim} below).

\section{Vanishing results}

In this section we prove our vanishing results, starting with the homology version. We fix a local ring $S$, a pair of exact zero-divisors $(x,y)$, and denote the local ring $S/(x)$ by $R$. It should be mentioned that the modules we consider in this section are not necessarily assumed to be finitely generated.

Since a (deleted) free resolution of $R$ over $S$ has the form
$$\cdots \to S \xrightarrow{y} S \xrightarrow{x} S \xrightarrow{y} S \xrightarrow{x} S \to 0$$
one has for any $R$-module $N$ the following
\[\tag{$\dagger$}
\Tor^S_q(R,N) \cong \left\{
\begin{array}{ll}
N & \text{for } q=0 \\
N/yN & \text{for } q>0 \text{ odd} \\
\Ann_N (y) & \text{for } q>0 \text{ even}
\end{array} \right.
\]
and
\[\tag{$\ddagger$}
\Ext_S^q(R,N) \cong \left\{
\begin{array}{ll}
N & \text{for } q=0 \\
\Ann_N (y) & \text{for } q>0 \text{ odd} \\
N/yN& \text{for } q>0 \text{ even}
\end{array} \right.
\]

Our main theorem on the vanishing of homology is the following.

\begin{theorem}\label{mainvanishinghomology}
Let $R=S/(x)$ where $S$ is a local ring and $(x,y)$ is a pair of exact zero-divisors in $S$. Furthermore, let $M$ and $N$ be $R$-modules. If there exists an integer $n \ge 2$ such that
$\Tor^{R}_{i}(M,N)=\Tor^R_i(M,N/yN)=\Tor^R_i(M,\Ann_N (y))=0$ for $1 \le i \le n$, then
$$\Tor^{S}_{i}(M,N) \cong \left \{
\begin{array}{ll}
M \otimes_S N & \text{for } i=0  \\
M \otimes_S N/yN & \text{for } 0<i<n \text{ and $i$ odd} \\
M \otimes_S \Ann_N (y) & \text{for } 0<i<n \text{ and $i$ even}.
\end{array} \right.$$
\end{theorem}

\begin{proof}
Consider the first quadrant change of rings spectral sequence \cite[Theorem 10.73]{Rot}
$$
\Tor_p^R \left ( M, \Tor_q^S(R,N) \right ) \underset{p}{\Longrightarrow} \Tor_{p+q}^S(M,N).
$$
\comments{$$
\xymatrixcolsep{.5in}\xymatrix{
\vdots & \vdots & \vdots & \iddots\\
E^2_{0,2}  & E^2_{1,2}  &  E^2_{2,2} & \cdots \\
E^2_{0,1}  & E^2_{1,1}  &  E^2_{2,1} \ar[llu]_{d^2_{2,1}} & \cdots \\
E^2_{0,0}  & E^2_{1,0}  &  E^2_{2,0} \ar[llu]_{d^2_{2,0}} & \cdots }
$$}
From ($\dagger$), the term $E^2_{p,q}$ is given by
$$E^2_{p,q} \cong \left \{
\begin{array}{ll}
\Tor^R_p(M,N) & \text{for } q=0 \\
\Tor^R_p(M,N/yN) & \text{for } q>0 \text{ odd} \\
\Tor^R_p(M, \Ann_N (y) ) & \text{for } q>0 \text{ even}.
\end{array} \right.$$
The vanishing assumptions imply that columns $1$ through $n$ of the $E^2$-page of this spectral sequence vanish, i.e., \ $E^2_{p,q} =0$ for all $q \in \mathbb{Z}$ and $1 \le p \le n$. Fixing such $p$ and $q$, we see that $E^{\infty}_{p,q}$ also vanishes since this term is a subquotient of $E^2_{p,q}$. Letting $H_i$ denote $\Tor^S_i(M,N)$ for all $i$, we have a filtration $\{\Phi^jH_i\}$ of
$H_i$ satisfying
$$0=\Phi^{-1}H_i\subseteq\Phi^0H_i\subseteq\cdots\subseteq\Phi^{i-1}H_i\subseteq\Phi^iH_i=H_i,$$ with $E^\infty_{j,i-j}\cong\Phi^jH_i/\Phi^{j-1}H_i$ for all $i$ and $j$.  Thus the vanishing of $E^\infty_{p,q}$ implies that $\Phi^p H_{p+q} = \Phi^{p-1} H_{p+q}$, that is, $\Phi^p H_q = \Phi^{p-1} H_q$ for all $q \in \mathbb{Z}$ and $1 \le p \le n$.

Now consider the zeroth column of the $E^2$-page. For a positive $q$, the $E^2_{0,q}$-term is isomorphic to $M \otimes_R N/yN$ when $q$ is odd, and isomorphic to $M \otimes_R \Ann_N (y)$ when $q$ is even. Since $E^2_{p,q} =0$ for all $q \in \mathbb{Z}$ and $1 \le p \le n$, there is an isomorphism $E^{\infty}_{0,q} \cong E^2_{0,q}$ for $q \le n-1$, giving
$$
E^{\infty}_{0,q} \cong \left \{
\begin{array}{ll}
M \otimes_S N & \text{for } q=0  \\
M \otimes_R N/yN & \text{for } 0<q<n \text{ and $q$ odd} \\
M \otimes_R \Ann_N (y) & \text{for } 0<q<n \text{ and $q$ even}.
\end{array} \right.
$$
But it follows from above that the equalities
$$
E^{\infty}_{0,q} \cong \Phi^0 H_q = \Phi^1 H_q = \cdots = \Phi^q H_q
$$
hold when $q<n$. Therefore, since $\Phi^q H_q = \Tor^S_q(M,N)$, we are done.
\end{proof}

As an immediate corollary we obtain the result, for the vanishing of homology, stated in the introduction.

\begin{corollary}\label{vanishinghomology}
Let $R=S/(x)$ where $S$ is a local ring and $(x,y)$ is a pair of exact zero-divisors in $S$. Furthermore, let $M$ and $N$ be $R$-modules such that $yN=0$. If there exists an integer $n \ge 2$ such that $\Tor^{R}_{i}(M,N)=0$ for $1 \le i \le n$, then $\Tor^{S}_{i}(M,N) \cong M \otimes_S N$ for $0 \le i \le n-1$.
Consequently, if $M$ and $N$ are nonzero and finitely generated, then $\Tor^{S}_{i}(M,N) \neq 0$ for $0 \le i \le n-1$.
\end{corollary}

Thus when the modules involved are finitely generated and nonzero, the corollary shows that the vanishing of homology over $R$ implies the \emph{non-vanishing} of homology over $S$.  This is in stark contrast to the case when $x$ is a regular element.

In certain cases we can show that the Tors over $S$ cannot vanish irrespective of vanishing of the Tors over $R$.

\begin{proposition} Let $R=S/(x)$ where $S$ is a local ring and $(x,y)$ is a pair of exact zero-divisors, both of which are \emph{minimal} generators of the maximal ideal of $S$. Furthermore, let $M$ and $N$ be nonzero finitely generated $R$-modules such that $yN=0$.  Then $\Tor^S_i(M,N)\ne 0$ for all $i\ge 0$.
\end{proposition}

\begin{proof} Consider a minimal free resolution of $M$ over $S$:
\[
F : \cdots \to F_2 \xrightarrow{\partial_2} F_1 \xrightarrow{\partial_1} F_0 \to 0
\]
Letting $m$ denote a minimal generator of $M$, we can define the homomorphism $f:S/(x)\to M$ sending $\bar 1$ to $m$.  Because $x$ and $y$ are minimal generators of the maximal ideal of $S$, we can lift this homomorphism to a chain map
\[
\xymatrix{
\cdots \ar[r] & S\ar[r]^{y}\ar[d]^{f_2} & S\ar[r]^{x}\ar[d]^{f_1} & S\ar[r]\ar[d]^{f_0} & S/(x)\ar[r]\ar[d]^{f} & 0\\
\cdots\ar[r] & F_2\ar[r]^{\partial_2} & F_1\ar[r]^{\partial_1} & F_0\ar[r] & M\ar[r] & 0
}
\]
in such a way that each $f_i$ is a split injection. Tensoring the entire diagram with $N$ we get the commutative diagram
\[
\xymatrix{
\cdots \ar[r] & N\ar[r]^{0}\ar[d]^{f_2\otimes N} & N\ar[r]^{0}\ar[d]^{f_1\otimes N} & N\ar[r]^{=}\ar[d]^{f_0\otimes N} & N\ar[r]\ar[d]^{f\otimes N} & 0\\
\cdots\ar[r] & F_2\otimes_SN\ar[r]^{\partial_2\otimes N} & F_1\otimes_SN\ar[r]^{\partial_1\otimes N} & F_0\otimes_SN\ar[r] & M\otimes_SN\ar[r] & 0
}
\]
Now let $n$ be a minimal generator of $N$.  Then $f_i(1)\otimes_S n$ is a minimal generator of $F_i\otimes_SN$, and by commutativity, is in $\ker(\partial_i\otimes_S N)$ for all $i\geq 1$. This element is not a boundary, however, since $\partial_{i+1}\subseteq \m F_i$, and no element in the image of
$\partial_{i+1}\otimes_SN$ is a minimal generator of $F_i\otimes_SN$. It follows that  $\Tor^S_i(M,N)\ne 0$ for all $i\ge 0$.
\end{proof}

We next state the cohomological versions of Theorem \ref{mainvanishinghomology} and Corollary \ref{vanishinghomology}; they are proved dually.

\begin{theorem}\label{mainvanishingcohomology}
Let $R=S/(x)$ where $S$ is a local ring and $(x,y)$ is a pair of exact zero-divisors in $S$. Furthermore, let $M$ and $N$ be $R$-modules. If there exists an integer $n \ge 2$ such that
$\Ext_{R}^{i}(M,N)=\Ext_R^i(M,N/yN)=\Ext_R^i(M,\Ann_N (y))=0$ for $1 \le i \le n$, then
$$\Ext_{S}^{i}(M,N) \cong \left \{
\begin{array}{ll}
\Hom_S(M,N) & \text{for } i=0 \\
\Hom_S(M, \Ann_N (y)) & \text{for } 0<i<n \text{ and $i$ odd} \\
\Hom_S(M,N/yN) & \text{for } 0<i<n \text{ and $i$ even}
\end{array} \right.$$
\end{theorem}

\begin{corollary}
Let $R=S/(x)$ where $S$ is a local ring and $(x,y)$ is a pair of exact zero-divisors in $S$. Furthermore, let $M$ and $N$ be $R$-modules with $yN=0$. If there exists an integer $n \ge 2$ such that
$\Ext_{R}^{i}(M,N)=0$ for $1 \le i \le n$, then
$\Ext_{S}^{i}(M,N) \cong \Hom_S(M,N)$ for $0\le i<n$.  Consequently, if $N=M\ne 0$, then
$\Ext_S^i(M,M)\ne 0$ for $0<i<n$.
\end{corollary}

\section{Complexity}
As in the previous section, we fix a local ring $S$, a pair of exact zero-divisors $(x,y)$, and denote the local ring $S/(x)$ by $R$. In this section all modules are assumed to be finitely generated.
Our aim is to compare free resolutions of modules over $R$ with those over $S$ and determine relationships involving complexities.

Given a local ring $A$ and an $A$-module $M$, there exists a (deleted) free resolution of $M$
$$\cdots \to F_2 \to F_1 \to F_0  \to 0$$
which is minimal, that is, it appears as a direct summand of every free resolution of $M$. The cokernel of the map $F_{n+1} \to F_n$ is the $n$th \emph{syzygy} module of $M$ and denoted by $\Omega_A^n(M)$. Minimal free resolutions are unique up to isomorphisms and hence the syzygies are uniquely determined up to isomorphism. Moreover, for every nonnegative integer $n$, the $n$th \emph{Betti number} $\beta^A_n(M) \stackrel{\text{def}}{=} \rank F_n$ is a well-defined invariant of $M$. It is well-known that $\dim_k \Ext_A^n(M,k) = \beta^A_n(M) = \dim_k \Tor^A_n(M,k)$ for every integer $n$ where $k$ is the residue field of $A$.
It is also clear that the projective dimension of $M$ is finite if and only if the Betti numbers of $M$ eventually vanish. Thus the asymptotic behavior of the Betti sequence $\beta^A_0(M), \beta^A_1(M), \beta^A_2(M), \dots$ determines an important homological property of $M$. Following ideas from modular representation theory \cite{Alp}, an invariant measuring how ``fast" the Betti sequence grows was introduced by Avramov in \cite{Avramov1} (cf. also \cite{Avramov2}). The \emph{complexity} of $M$, denoted by $\cx_A(M)$, is defined as
$$\cx_A(M) \stackrel{\text{def}}{=} \inf \{ t \in \mathbb{N} \cup \{ 0 \} \mid \exists a \in \mathbb{R} \text{ such that } \beta^A_n(M) \le an^{t-1} \text{ for all } n \},$$
and measures the polynomial rate of growth of the Betti sequence of $M$. It follows from the definition that $M$ has finite projective dimension if and only if $\cx_A(M) =0$, whereas $\cx_A(M) =1$ if and only if the Betti sequence of $M$ is bounded. For an arbitrary local ring, the complexity of a module is not necessarily finite \cite[4.2.2]{Av3}. In fact, by \cite[Theorem 2.3]{Gulliksen2}, finiteness of complexity for all finitely generated $A$-modules is equivalent to $A$ being a complete intersection.

We now return to our previous setting of exact zero-divisors.  We first remark that every nonzero $R$-module has infinite projective dimension over $S$, that is, every such module has positive complexity over $S$.
Indeed, ($\dagger$) from the second paragraph of Section 2 shows that if $\Tor^S_i(R,M)=0$ for all $i\gg 0$, then $M/yM=0$.  Thus $M=0$ by Nakayama's Lemma.

Over a local ring $A$, the complexity of a module equals the complexity of any of its syzygies: their minimal free resolutions are the same except at the beginning. Moreover, given a short exact sequence
$$
0 \to M_1 \to M_2 \to M_3 \to 0
$$
of $A$-modules, the inequality
\begin{equation}\label{cxineq}
\cx_A(M_u) \le \max \{ \cx_A(M_v), \cx_A(M_w) \}
\tag{*}
\end{equation}
holds for $\{ u,v,w \} = \{ 1,2,3 \}$. This follows simply by comparing the $k$-vector space dimensions of the Tor modules in the long exact sequence
$$\cdots \to \Tor^A_n(M_1,k) \to \Tor^A_n(M_2,k) \to \Tor^A_n(M_3,k) \to \Tor^A_{n-1}(M_1,k) \to \cdots,$$
where $k$ is the residue field of $A$.

In the next proposition we use the inequality (\ref{cxineq}) and prove that if $M$ is an $R$-module with $\cx_S(M) \neq 1$, then $\cx_S(M) = \cx_S(\Omega_R^n(M))$ for all $n$. Here the assumption $\cx_S(M) \neq 1$ is necessary: the $S$-module $R$ has a minimal free resolution
$$\cdots \to S \xrightarrow{y} S \xrightarrow{x} S \xrightarrow{y} S \xrightarrow{x} S \to 0 $$ and hence has complexity one over $S$. However its syzygies $\Omega_R^n(R)$ are all zero for $n>0$.

\begin{proposition}\label{samecomplexity}
Let $R=S/(x)$ where $S$ is a local ring and $x$ is an exact zero-divisor in $S$. Then, for every finitely generated $R$-module $M$ with $\cx_S(M) \neq 1$, the equality $\cx_S(M) = \cx_S (\Omega_R^n(M))$ holds for all $n$.
\end{proposition}

\begin{proof}
If $\cx_{S}(M)=0$, then $M=0$ (see the third paragraph of Section 3.) Thus the result is trivial in this case. Next suppose $\cx_S(M)>1$. Consider the short exact sequence
$$0 \to \Omega_R^1(M) \to F \to M \to 0$$
where $F$ is a free $R$-module. Since the $S$-module $R$ has complexity one so does $F$.  Hence
the result follows from the inequality $(\ref{cxineq})$ and the short exact sequence considered above. \end{proof}

Next we will compare the Betti numbers and complexities of modules over $R$ with those over $S$. For that we first set some notations that generalize the notion of the Betti number and the complexity of a module.

Let $(A,\m)$ be a local ring with residue field $k$, and $M$ and $N$ be $A$-modules with the property that $M \otimes_A N$ has finite length. Then, for every nonnegative integer $n$, the length of $\Tor^A_n(M,N)$ is finite. We define this length to be the $n$th Betti number $\beta^A_n(M,N)$ of the pair $(M,N)$, that is, $\beta^A_n(M,N)\stackrel{\text{def}}{=} \ell\left(\Tor^A_n(M,N)\right)$. The \emph{length complexity of the pair} $(M,N)$, denoted by $\lcx_A (M,N)$, is then defined as:
$$\lcx_A (M,N) \stackrel{\text{def}}{=} \inf \{ t \in \mathbb{N} \cup \{ 0 \} \mid \exists a \in \mathbb{R} \text{ such that } \beta^A_n(M,N) \le an^{t-1} \text{ for all } n \},$$
cf. \cite[the discussion preceding Definition 2.1]{Dao2}.
Although letting $N=k$, we recover the Betti number and the ordinary complexity of $M$, that is, 
\[
\beta^{A}_{n}(M) = \beta^{A}_{n}(M,k) \text{ and } \cx_{A}(M) = \lcx_{A}(M,k), 
\]
our definition for $\lcx_{A}(M,N)$ of the pair $(M,N)$ is different than the one originally defined by Avramov and Buchweitz \cite{AvBu}, where the minimal number of generators of the cohomology modules $\Ext^n_A(M,N)$ is used.  In general there is no comparison between these two definitions of Betti numbers of the pair $(M,N)$; see also \cite[Theorem 5.4]{Dao2}.

\begin{theorem}\label{bettipair}
Let $R=S/(x)$ where $S$ is a local ring and $(x,y)$ is a pair of exact zero-divisors in $S$. Furthermore, let $M$ and $N$ be finitely generated $R$-modules such that $yN=0$ and $M \otimes_R N$ has finite length. Then, for all $n$,
\begin{equation}\label{bettieq}
\beta^R_n(M,N) - \sum_{i=0}^{n-2} \beta^R_i(M,N) \le \beta^S_n(M,N) \le \sum_{i=0}^n \beta^R_i(M,N)
\tag{\ref{bettipair}}
\end{equation}
\end{theorem}

\begin{remark}
We have used the convention that negative Betti numbers are zero.
\end{remark}

\begin{proof}
As in the proof of Theorem \ref{mainvanishinghomology}, we consider the first quadrant change of rings spectral sequence:
$$\Tor_p^R \left ( M, \Tor_q^S(R,N) \right ) \underset{p}{\Longrightarrow} \Tor_{p+q}^S(M,N).$$
Since $yN=0$, we see from ($\dagger$) that the $E^2$-page entries are given by $E^2_{p,q} = \Tor^R_p(M,N)$.

We first prove the left-hand inequality of \eqref{bettieq}. Fix an integer $n$ and consider the short exact sequence
$$0 \to \Phi^{n-1}H_n \to \Tor^S_n(M,N) \to E^{\infty}_{n,0} \to 0,$$
where $\Phi^iH_n$ is the filtration of $H_n$ from the proof of Theorem \ref{mainvanishinghomology}.
Since $E^{\infty}_{n,0} = \Ker d^n_{n,0}$, we obtain the inequality $\ell \left ( \Tor^S_n(M,N) \right ) \ge \ell ( \Ker d^n_{n,0} )$. Now for all $2 \le p \le n$, there is an exact sequence
$$0 \to \Ker d^p_{n,0} \to \Ker d^{p-1}_{n,0} \to \Im d^p_{n,0} \to 0,$$
which implies
\begin{eqnarray*}
\ell ( \Ker d^n_{n,0} ) & = & \ell ( \Ker d^{n-1}_{n,0} ) - \ell ( \Im d^n_{n,0} ) \\
& = & \ell ( \Ker d^{n-2}_{n,0} ) - \left ( \ell ( \Im d^{n-1}_{n,0} ) + \ell ( \Im d^n_{n,0} ) \right ) \\
& \vdots & \\
& = & \ell ( \Ker d^1_{n,0} ) - \sum_{i=2}^n \ell ( \Im d^i_{n,0} ).
\end{eqnarray*}
For $2 \le i \le n$, the image of $d^i_{n,0}$ is a submodule of $E^i_{n-i,i-1}$, and the latter is a subquotient of $E^2_{n-i,i-1}$. Then since $E^2_{n-i,i-1} = \Tor^R_{n-i}(M,N)$, there is an inequality $\ell ( \Im d^i_{n,0} ) \le \ell \left ( \Tor^R_{n-i}(M,N) \right )$. Moreover the module $E^2_{n,0}$ is a subquotient of $\Ker d^1_{n,0}$. Thus, since $E^2_{n,0} = \Tor^R_{n}(M,N)$, we have $\ell ( \Ker d^1_{n,0} ) \ge \ell \left ( \Tor^R_{n}(M,N) \right )$. This gives
\begin{eqnarray*}
\ell \left ( \Tor^S_n(M,N) \right ) & \ge &  \ell ( \Ker d^n_{n,0} ) \\
& = & \ell ( \Ker d^1_{n,0} ) - \sum_{i=2}^n \ell ( \Im d^i_{n,0} ) \\
& \ge & \ell \left ( \Tor^R_{n}(M,N) \right ) - \sum_{i=0}^{n-2} \ell \left ( \Tor^R_{i}(M,N) \right ),
\end{eqnarray*}
proving the left-hand inequality.

For the right-hand inequality, we fix an integer $n$ and consider the short exact sequence:
$$
0 \to \Phi^{p-1}H_n \to \Phi^{p}H_n \to E^{\infty}_{p,n-p} \to 0
$$
for $0 \le p \le n$. Counting the lengths, we obtain equalities
\begin{eqnarray*}
\ell \left ( \Phi^{n}H_n \right ) & = & \ell \left ( \Phi^{n-1}H_n \right ) + \ell \left ( E^{\infty}_{n,0} \right ) \\
& = & \ell \left ( \Phi^{n-2}H_n \right ) + \ell \left ( E^{\infty}_{n-1,1} \right ) + \ell \left ( E^{\infty}_{n,0} \right ) \\
& \vdots &\\
& = & \sum_{i=0}^n \ell \left ( E^{\infty}_{i,n-i} \right ).
\end{eqnarray*}
Each $E^{\infty}_{i,n-i}$ is a subquotient of $E^{2}_{i,n-i}$, and so since $E^2_{i,n-i} = \Tor^R_{i}(M,N)$, we obtain the inequality $\ell ( E^{\infty}_{i,n-i} ) \le \ell \left ( \Tor^R_{i}(M,N) \right )$. Then since $\Phi^{n}H_n = \Tor^S_{n}(M,N)$, we obtain
\begin{eqnarray*}
\ell \left ( \Tor^S_{n}(M,N) \right ) & = & \ell \left ( \Phi^{n}H_n \right ) \\
& = & \sum_{i=0}^n \ell \left ( E^{\infty}_{i,n-i} \right ) \\
& \le & \sum_{i=0}^n \ell \left ( \Tor^R_{i}(M,N) \right ),
\end{eqnarray*}
proving the right-hand inequality.
\end{proof}

As a consequence, using the right-hand side of the inequality \eqref{bettieq}, we obtain an upper bound for $\lcx_S(M,N)$ in terms of the complexity of $(M,N)$ over $R$.

\begin{corollary}\label{complexitypair}
Let $R=S/(x)$ where $S$ is a local ring and $(x,y)$ is a pair of exact zero-divisors in $S$. Furthermore, let $M$ and $N$ be finitely generated $R$-modules such that $yN=0$ and $M \otimes_R N$ has finite length. Then $\lcx_S(M,N) \le \lcx_R(M,N)+1$.
\end{corollary}

\begin{proof}
If $\lcx_R(M,N)=\infty$, then there is nothing to prove. So suppose $\lcx_R(M,N)= c<\infty$. Then, by the definition, there exists a real number $a$ such that $\beta^R_n(M,N) \le an^{c-1}$ for all $n$. By Theorem \ref{bettipair}, the inequality
$$\beta^S_n(M,N) \le \sum_{i=0}^n \beta^R_i(M,N) \le \sum_{i=0}^n ai^{c-1} \le (n+1)an^{c-1}$$
holds for all $n$. Therefore there is a real number $b$ such that $\beta^S_n(M,N) \le bn^{c}$ for all $n$. This shows that $\lcx_S(M,N) \le c+1$.
\end{proof}

We are unaware of an example of a pair of $R$-modules for which equality holds in the left-hand side of \eqref{bettieq}. On the other hand, equality may occur in the right-hand side. Indeed, when the exact zero-divisors $x$ and $y$ are minimal generators of the maximal ideal of $S$, Henriques and {\c{S}}ega prove \cite[1.7]{HenriquesSega} that the equality
$$
\sum_{n=0}^{\infty} \beta^S_n(M) t^n = \frac{1}{1-t} \sum_{n=0}^{\infty} \beta^R_n(M) t^n
$$
of Poincar{\'e} series holds for every finitely generated $R$-module $M$. This gives:
$$
\beta^S_n(M) = \sum_{i=0}^n \beta^R_i(M)
$$
However, when $x$ and $y$ are arbitrary, the equality of the Poincar{\'e} series stated above may fail:
\begin{example} Let $S=k[[x]]/(x^{3})$ where $k$ is a field. Then $x^{2}$ is an exact zero divisor in $S$. Set $R=S/(x^{2})\cong k[[x]]/(x^{2})$.
It can be seen that:
$$
\sum_{n=0}^{\infty} \beta^S_n(k) t^n = \frac{1}{1-t}= \sum_{n=0}^{\infty} \beta^R_n(k) t^n
$$
\end{example}
\noindent This example also shows that the inequality of Corollary \ref{complexitypair} can be strict.

We  now give an example illustrating the fact that the left-hand inequality of \eqref{bettieq}
does give useful lower bounds in some cases:

\begin{example}\label{lhbound} Let $R=k[x_1\dots,x_e]/(x_1,\dots,x_e)^2$, and $M$ be a
finitely generated $R$-module.  Then
$\Omega_R^1(M)$ is a finite dimensional vector space over $k$ of dimension $\beta_1^R(M)$.
It is easy to see that the Betti numbers of $k$ are $\beta_n^R(k)=e^n$.  It follows that
$\beta_n^R(M)=\beta_1^R(M)e^{n-1}$ for all $n\ge 1$. From the left-hand inequality of \eqref{bettieq}
we have
\begin{eqnarray*}
\beta^S_n(M) & \ge & \beta_n^R(M)-\sum_{i=0}^{n-2}\beta_i^R(M) \\
& = & \beta_1^R(M)e^{n-1}-\left(\sum_{i=1}^{n-2}\beta_1^R(M)e^{i-1}\right)-\beta_0^R(M)\\
& = & \beta_1^R(M)\left(e^{n-1}-\frac{e^{n-2}-1}{e-1}\right)-\beta_0^R(M)\\
& = & \beta_1^R(M)\left(\frac{e^{n}-e^{n-1}-e^{n-2}+1}{e-1}\right)-\beta_0^R(M)\\
& \ge & \frac{\beta_1^R(M)}{2}e^{n-1}-\beta_0^R(M)
\end{eqnarray*}
for $e\ge 2$ and for any ring $S$ such that there exists an exact zero-divisor $x$ with $R\cong S/(x)$.
Note that the last inequality follows since for $e\ge 2$ we have $e^2-e-2\ge 0$.  Then $e^{n-2}(e^2-e-2)\ge 0$, which implies that $e^n-e^{n-1}-2e^{n-2}+2\ge 0$.  Thus $2(e^n-e^{n-1}-e^{n-2}+1)\ge e^{n-1}(e-1)$, and the desired inequality follows. In particular, $R$-modules must have exponential growth over $S$ as well.
As a specific example, let $S=k[x,y,z]/(x^2,y^2,z^2,yz)$.  Then $x$ is an exact zero-divisor in $S$,
and $R=S/(x)\cong k[y,z]/(y,z)^2$ has the form above.
\end{example}

When $N=k$, the assumptions that $M \otimes N$ has finite length and $yN=0$ hold automatically. Therefore, in this situation, Theorem \ref{bettipair} and Corollary \ref{complexitypair} can be summarized as follows:

\begin{corollary}\label{complexitymodule}
Let $R=S/(x)$ where $S$ is a local ring and $x$ is an exact zero-divisor in $S$. Then, for every finitely generated $R$-module $M$, the inequalities
$$\beta^R_n(M) - \sum_{i=0}^{n-2} \beta^R_i(M) \le \beta^S_n(M) \le \sum_{i=0}^n \beta^R_i(M)$$
hold for all $n$. Consequently $\cx_S(M) \le \cx_R(M)+1$ holds.
\end{corollary}

\begin{remark} It follows from \cite[4.4]{AvramovHenriquesSega} that $R$ is a complete intersection if and only if $S$ is a complete intersection. The complexity inequality obtained in Corollary \ref{complexitymodule} gives a different proof for the `only if' direction this result: if $\cx_{R}(k)<\infty$, where $k$ is the residue field of $R$, then it follows from Corollary \ref{complexitymodule} that $\cx_{S}(k)<\infty$ and hence, by \cite[2.5]{Gulliksen2}, $S$ is a complete intersection.

Another observation related to the result stated above concerns commutative local Cohen-Macaulay Golod rings \cite[5.2]{Av3}. Assume $S$ is such a ring. Since a finitely generated module has infinite complexity over $S$ in case it has infinite projective dimension over $S$ and $\text{codepth}(S) \geq 2$ \cite[5.3.3(2)]{Av3}, we conclude $\text{codepth}(S) \leq 1$ (Recall $\cx_{S}(R)=1$). Moreover, as $x$ is not regular, $\text{codepth}(S)=1$. This implies that $S$ is a hypersurface and hence $R$ is a complete intersection.
\end{remark}

As discussed in the introduction, when $x$ is regular the complexity inequality is quite different than the one obtained in Corollary \ref{complexitymodule}. More precisely, in that case the inequalities $\cx_S(M) \le \cx_R(M) \le \cx_S(M)+1$ hold. In particular the complexity of $M$ over $R$ is finite if and only if it is finite over $S$. However, in our situation, when $x$ is an exact zero-divisor, we are unable to deduce any further inequalities, such as $\cx_R(M) \le \cx_S(M)$, from Theorem \ref{bettipair}. In fact we do not know whether there exists an $R$-module $M$ with $\cx_S(M)<\infty$ and $\cx_R(M)=\infty$. We record this in the next question.

\begin{question} Let $R=S/(x)$ where $S$ is a local ring and $x$ is an exact zero-divisor in $S$. Is $\cx_R(M) \le \cx_S(M)$ for all finitely generated $R$-modules $M$?
\end{question}

\section{Canonical elements of $\Ext^2_R(M,M)$ and Lifting}

In this section we restrict our attention to the case where $(x,x)$ is a pair of exact zero-divisors in
the local ring $S$, and $R=S/(x)$. We discuss natural chain endomorphisms of complexes over $R$, following the construction in \cite[Section 1]{Eisenbud}, and show that whether or not they are null-homotopic  dictates the liftability of $R$-modules to $S$.  These results generalize classical results (see, for example \cite{ADS}) for lifting modules modulo a regular element to modulo the square of the regular element.

\subsection*{Canonical endomorphisms of complexes} Let
\begin{equation}\label{F}
F:\cdots \to F_{i+1}\xrightarrow{\partial_{i+1}} F_i \xrightarrow{\partial_i} F_{i-1} \to\cdots
\end{equation}
be a complex of finitely generated free $R$-modules.  We let
\begin{equation}\label{preimage}
\widetilde F:\cdots \to\widetilde F_{i+1}\xrightarrow{\widetilde\partial_{i+1}} \widetilde F_i \xrightarrow{\widetilde\partial_i} \widetilde F_{i-1} \to\cdots
\end{equation}
denote a preimage over $S$ of the complex $F$, that is, a sequence of homomorphisms
$\widetilde\partial_i:\widetilde F_i \to \widetilde F_{i-1}$ of free $S$-modules
such that $F$ and $\widetilde F \otimes_S R$ are isomorphic $R$-complexes.  From the fact that $\widetilde\partial_{i-1}\widetilde\partial_i(\widetilde F_i)\subseteq x \widetilde F_{i-2}$ for all $i$,
we can write
\begin{equation}\label{factorization}
\widetilde\partial_{i-1}\widetilde\partial_i = x\widetilde s_i
\end{equation}
for some homomorphism $\widetilde s_i:\widetilde F_i \to \widetilde F_{i-2}$.
Now we define the homomorphisms $s_i:F_i \to F_{i-2}$ by
\begin{equation}\label{s}
s_i=\widetilde s_i \otimes_S R
\end{equation}
for all $i$.

There are several properties of the $s_i$ which we should like to mention.  See \cite[Section 1]{Eisenbud}
for the proofs. (Note that in our case $(x)/(x)^2=(x)\cong S/(x)$ is a free $S/(x)$-module.)

\begin{enumerate}
\item[(a)] The definition of $s_i$ is independent of the factorization in (\ref{factorization}).
\comments{\begin{proof} Suppose that $\widetilde\partial_{i-1}\widetilde\partial_i = x\widetilde t_i$
is another factorization.  Then $x(\widetilde s_i- \widetilde t_i)=0$.  Since $\Ann_S(x)=(x)$
it follows that $\widetilde s_i - \widetilde t_i = xu_i$ for some $u_i:\widetilde F_i \to \widetilde F_{i-2}$.  Thus, modulo $x$, we have $s_i=t_i$.
\end{proof}}
\item[(b)] The family $s=\{s_i\}$ is a chain endomorphism of $F$ of degree $-2$.
\comments{\begin{proof} For each $i$ we have
\[
x(\widetilde s_{i-1}\widetilde\partial_i-\widetilde\partial_{i-2}\widetilde s_i)=
\widetilde\partial_{i-2}\widetilde\partial_{i-1}\widetilde\partial_i-
\widetilde\partial_{i-2}\widetilde\partial_{i-1}\widetilde\partial_i=0.
\]
It follows that $\widetilde s_{i-1}\widetilde\partial_i-\widetilde\partial_{i-2}\widetilde s_i=xu_i$
for some $u_i:\widetilde F_i \to \widetilde F_{i-3}$.  Hence
$s_{i-1}\partial_i=\partial_{i-2}s_i$.
\end{proof}}
\item[(c)] \label{naturality} Let
\[
G: \cdots \to G_{i+1}\xrightarrow{\delta_{i+1}} G_i \xrightarrow{\delta_i} G_{i-1} \to \cdots
\]
be another complex of finitely generated free $R$-modules, and assume that there exists a chain map $f:F \to G$.
Let $t=\{t_i=\widetilde t_i\otimes_S R:G_i\to G_{i-2}\}$ be the chain map defined by the factorizations
$\widetilde\delta_{i-1}\widetilde\delta_i=x\widetilde t_i$ for all $i$, where $\widetilde G$ is a preimage over $S$ of $G$. Then the chain maps $fs$ and $tf$ are homotopic.
\comments{\begin{proof} Assume that the degree of $f$ is $d$.  For each $i$ let
$\widetilde f_i:\widetilde F_i\to \widetilde G_{i+d}$ be a preimage of $f_i:F_i \to G_{i+d}$ over $S$.
Since $f_{i-1}\partial_i=\delta_{i+d}f_i$ for all $i$ there exist
$\widetilde h_i:\widetilde F_i\to \widetilde G_{i+d-1}$ such that
$\widetilde f_{i-1}\widetilde\partial_i-\widetilde\delta_{i+d}\widetilde f_i=x\widetilde h_i$. Now for all $i$ we have
\begin{eqnarray*}
x(\widetilde f _{i-2}\widetilde s_i-\widetilde t_{i+d}\widetilde f_i) & = &
\widetilde f _{i-2}\widetilde\partial_{i-1}\widetilde\partial_i-
\widetilde\delta_{i+d-1}\widetilde\delta_{i+d}\widetilde f_i\\
& = & (\widetilde\delta_{i+d-1}\widetilde f_{i-1}+x\widetilde h_{i-1})\widetilde\partial_i
-\widetilde\delta_{i+d-1}(\widetilde f_{i-1}\widetilde\partial_i-x\widetilde h_i)\\
& = & x(\widetilde h_{i-1}\widetilde\partial_i+\widetilde\delta_{i+d-1}\widetilde h_i).
\end{eqnarray*}
It follows that there exists for each $i$ homomorphisms $u_i: \widetilde F_i \to \widetilde G_{i+d-2}$
such that
\[
(\widetilde f _{i-2}\widetilde s_i-\widetilde t_{i+d}\widetilde f_i)-
(\widetilde h_{i-1}\widetilde\partial_i+\widetilde\delta_{i+d-1}\widetilde h_i)=xu_i.
\]
Thus for all $i$ we have
\[
f_{i-2}s_i-t_{i+d}f_i=h_{i-1}\partial_i+\delta_{i+d-1}h_i
\]
where $h_i=\widetilde h_i\otimes_S R$ for all $i$.  This shows that $fs$ and $tf$ are homotopic.
\end{proof}}
\item[(d)] From (c) it follows that the definition of the $s_i$ is independent, up to homotopy, of the preimage $\widetilde F$ of $F$ chosen in (\ref{preimage}).
\end{enumerate}

\subsection*{The group $\Ext^n_A(M,M)$}  Let $A$ be an associative ring, and $M$ an $A$-module.
Suppose that $F$ is a projective resolution of $M$.
Then $\Ho^n(\Hom_A(F,F))$ is the group of homotopy equivalence classes of chain endomorphisms of $F$ of degree $n$.  For a chain endomorphism $s$ of $F$ of degree $n$, we let $[s]$ denote the class of $s$ in $\Ho^n(\Hom_A(F,F))$. Let $G$ be another projective resolution of $M$ over $A$.  Then the comparison maps $f:F \to G$ and $g:G\to F$ lifting the identity map on $M$ are homotopically equivalent. That is, $fg$ is homotopic to the identity map on $G$ and $gf$ is homotopic to the identity map on $F$.  It follows that the
map
\begin{equation}\label{iso}
\theta^F_G:\Ho^n(\Hom_A(F,F)) \to \Ho^n(\Hom_A(G,G))
\end{equation}
given by $[s]\mapsto [fsg]$ is an isomorphism, with
inverse $\theta^G_F:[s]\mapsto [gsf]$.  It is well-known that this group is $\Ext^n_A(M,M)$ (see, for example \cite{AV}.)

\subsection*{Canonical elements of $\Ext^2_R(M,M)$}
Returning to the situation where $R=S/(x)$ for the pair $(x,x)$ of exact zero-divisors, let $F$ be a free resolution of $M$ over $R$, and $s$ be the endomorphism of $F$ defined by (\ref{s}).  Thus we have the element $[s]\in\Ho^2(\Hom_R(F,F))$.  That we call $[s]$ a canonical element of $\Ext^2_R(M,M)$ is reinforced by the following lemma.

\begin{lemma}\label{canonicallemma}
Let $R=S/(x)$ where $S$ is a local ring and $(x,x)$ is a pair of
exact zero-divisors in $S$.  Suppose that $F$ and $G$ are free resolutions of a finitely generated module $M$ over $R$, that $s$ is the canonical endomorphism of $F$ as defined in (\ref{s}), and that $t$ is the canonical endomorphism of $G$ as defined in (\ref{s}).  Then we have \[
\theta^G_F([t])=[s]
 \]
where  $\theta^G_F$ is the isomorphism defined in (\ref{iso}).
\end{lemma}

\begin{proof} First assume that $F$ is a minimal free resolution of $M$.  Then the comparison map
$f:F\to G$ lifting the identity map on $M$ can be chosen to be a split injection, with splitting
$g:G\to F$, also lifting the identity map on $M$.  In particular, we have $gf=\Id_F$, the identity
map on $F$.

Denote the differential on $F$ by $\partial$, and that on $G$ by $\delta$.  Let $(\widetilde F,\widetilde\partial)$ be a preimage over $S$ of $(F,\partial)$, and $(\widetilde G,\widetilde\delta)$ be a preimage
over $S$ of $(G,\delta)$.  We choose preimages $\widetilde f$ of $f$ and $\widetilde g$ of $g$ over $S$ such that $\widetilde g\widetilde f=\Id_{\widetilde F}$.

As $g_{i-1}\delta_i=\partial_ig_i$ for all $i$, there exists
$u_i:\widetilde G_i\to \widetilde F_{i-1}$ such that
$\widetilde g_{i-1}\widetilde\delta_i=\widetilde\partial_i\widetilde g_i+xu_i$ for all $i$.
Similarly, there exists $v_i:\widetilde F_i\to\widetilde G_{i-1}$ such that
$\widetilde\delta_i\widetilde f_i=\widetilde f_{i-1}\widetilde\partial_i+xv_i$ for all $i$.
Thus we have
\begin{eqnarray*}
x(\widetilde g _{i-2}\widetilde t_i\widetilde f_i-\widetilde s_i) & = &
\widetilde g _{i-2}\widetilde\delta_{i-1}\widetilde\delta_i\widetilde f_i-\widetilde\partial_{i-1}\widetilde\partial_i\\
& = & (\widetilde\partial_{i-1}\widetilde g_{i-1}+xu_{i-1})(\widetilde f_{i-1}\widetilde\partial_i+xv_i)-
\widetilde\partial_{i-1}\widetilde\partial_i\\
& = & x(\widetilde\partial_{i-1}\widetilde g_{i-1}v_i+u_{i-1}\widetilde f_{i-1}\widetilde\partial_i).
\end{eqnarray*}
It follows that $g_{i-2}t_if_i-s_i=\partial_{i-1}(g_{i-1}\overline v_i)+(\overline u_{i-1}f_{i-1})\partial_i$ for all $i$, where
$\overline u_i=u_i\otimes_SR$ and $\overline v_i=v_i\otimes_S R$.
We will have shown that $gtf$ is homotopic to $s$ with homotopy $h_i=\overline u_if_i$ once we know that $\overline u_if_i=g_{i-1}\overline v_i$ for all $i$.  But this is easy:
\begin{eqnarray*}
x(u_i\widetilde f_i-\widetilde g_{i-1}v_i) & = &
(\widetilde g_{i-1}\widetilde\delta_i-\widetilde\partial_i\widetilde g_i)\widetilde f_i-
\widetilde g_{i-1}(\widetilde\delta_i\widetilde f_i-\widetilde f_{i-1}\widetilde\partial_i)\\
& = & -\widetilde\partial_i\widetilde g_i\widetilde f_i+\widetilde g_{i-1}\widetilde f_{i-1}\widetilde\partial_i\\
& = & 0,
\end{eqnarray*}
hence the claim follows.

Notice that we also have $\theta^F_G([s])=[t]$, when $F$ is minimal.  Therefore, for two arbitrary free resolutions $F$
and $G$ of $M$, that $\theta^G_F([t])=[s]$ follows from composing $\theta^G_F=\theta^G_L\theta^L_F$
where $L$ is a minimal free resolution of $M$.
\end{proof}

\subsection*{Lifting}
Let $B$ be an associative ring, $I$ an ideal of $B$, and $A=B/I$.
Recall that a finitely generated $A$-module $M$ is said to \emph{lift} to $B$, with \emph{lifting} $M'$,
if there exists a finitely generated $B$-module $M'$ such that $M\cong M'\otimes_B A$, and
$\Tor^B_i(M',A)=0$ for all $i\ge 1$.  Similarly, a complex of finitely generated free $A$-modules
\[
F: \cdots \to F_{i+1}\xrightarrow{\partial_{i+1}} F_i \xrightarrow{\partial_i} F_{i-1} \to\cdots
\]
is said to \emph{lift} to $B$, with \emph{lifting} $\widetilde F$,
if there exists a preimage $\widetilde F$ of $F$
\[
\widetilde F:\cdots \to\widetilde F_{i+1}\xrightarrow{\widetilde\partial_{i+1}} \widetilde F_i \xrightarrow{\widetilde\partial_i} \widetilde F_{i-1} \to\cdots
\]
such that $\widetilde\partial_{i-1}\widetilde\partial_i=0$ for all $i$. A close connection between these two notions of lifting will be explained in the next theorem.
We want also to show that when $R=S/(x)$ for $(x,x)$ a pair of exact zero-divisors, the triviality of the canonical element $[s]$ determines whether the module $M$ lifts to $S$.

\begin{theorem}\label{thmlift}
Let $R=S/(x)$ where $S$ is a local ring and $(x,x)$ is a pair of exact zero-divisors in $S$. Then for every finitely generated $R$-module $M$, the following are equivalent.
\begin{enumerate}
\item $M$ lifts to $S$.
\item The canonical element $[s]$ in $\Ext^2_R(M,M)$ is trivial.
\item Every free resolution of $M$ by finitely generated free $R$-modules lifts to $S$.
\item Some free resolution of $M$ by finitely generated free $R$-modules lifts to $S$.
\end{enumerate}
\end{theorem}

\begin{proof} $(1)\implies(2)$.  Suppose that $M'$ is a lifting of $M$ to $S$.  Let
\[
\widetilde F:\cdots \to\widetilde F_{2}\xrightarrow{\widetilde\partial_{2}} \widetilde F_1 \xrightarrow{\widetilde\partial_1} \widetilde F_{0}\to 0
\]
be a resolution of $M'$ by finitely generated free $S$-modules.  Since $\Tor^S_i(M',R)=0$
for all $i>0$, $F=\widetilde F \otimes_S R$ is a resolution of $M\cong M'\otimes_S R$
by finitely generated free $R$-modules.  Computing the endomorphism $s$ from the preimage
$\widetilde F$ of $F$, which is exact, we see that $s$ is actually the zero endomorphism,
and is therefore certainly trivial in $\Ext^2_R(M,M)$.

$(2)\implies(3)$. By Lemma \ref{canonicallemma} the canonical element of $\Ext^2_R(M,M)$
is trivial regardless of which resolution by finitely generated free $R$-modules $F$ of $M$
we choose to define it.  Therefore let $F$ be an arbitrary such resolution of $M$, and let $s$ be the canonical chain endomorphism defined as in (\ref{s}) of the subsection above on canonical endomorphisms of complexes.  By assumption $s$ is homotopic to zero.
Therefore there exists a homotopy $h=\{h_i\}$ with $h_i:F_i\to F_{i-1}$ such that
$s_i=\partial_{i-1} h_i + h_{i-1}\partial_i$ for all $i$.  Let $\widetilde F$
be an arbritrary preimage of $F$, with maps $\widetilde\partial$.
Let $\widetilde h_i:\widetilde F_i\to\widetilde F_{i-1}$
be a preimage of $h_i$ for all $i$. There exists $u_i:\widetilde F_i\to \widetilde F_{i-2}$ such that $\widetilde s_i=\widetilde \partial_{i-1} \widetilde h_i + \widetilde h_{i-1}\widetilde\partial_i+xu_i$ for all $i$. Now consider the preimage  $F^\sharp$ of $F$ where we take
$F^\sharp_i=\widetilde F_i$ for all $i$, but we take the maps $\partial^\sharp_i=\widetilde\partial_i-x\widetilde h_i$ instead.
we have
\begin{eqnarray*}
\partial^\sharp_{i-1}\partial^\sharp_i & = & (\widetilde\partial_{i-1}-x\widetilde h_{i-1})(\widetilde\partial_i-x\widetilde h_i)\\
& = & \widetilde\partial_{i-1}\widetilde\partial_i-x(\widetilde\partial_{i-1}\widetilde h_i + \widetilde h_{i-1}\widetilde\partial_i)\\
& = & x(\widetilde s_i -\widetilde\partial_{i-1}\widetilde h_i - \widetilde h_{i-1}\widetilde\partial_i))\\
& = & 0.
\end{eqnarray*}
Thus $F^\sharp$ is a lifting of $F$ to $S$.

$(3)\implies (4)$ is trivial.  To show that $(4) \implies (1)$, assume that $F$ is a free resolution
of $M$ by finitely generated free $R$-modules, which lifts to the complex $\widetilde F$ over $S$.
We claim that $\Ho_i(\widetilde F)=0$ for $i\ne 0$.  Indeed, if $\widetilde\partial_{i}(a)=0$, some
$a\in \widetilde F_i$, then $a=\widetilde\partial_{i+1}(b)+xc$ for some $b\in\widetilde F_{i+1}$ and
$c\in\widetilde F_i$ by exactness of $F$.  Since $x\widetilde\partial_{i}(c)=0$, we have $\widetilde\partial_i(c)\in x\widetilde F_{i-1}$. Again by exactness of $F$ we have
$c=\widetilde\partial_{i+1}(d)+xe$ for some $d\in\widetilde F_{i+1}$ and $e\in\widetilde F_i$.
Therefore $a=\widetilde\partial_{i+1}(b+xd)$. It follows that $\widetilde F$ is a
resolution of $M'=\Ho_0(\widetilde F)$ by finitely generated free $S$-modules, and thus
$M'$ is a lifting of $M$ to $S$.
\end{proof}

We end with an example showing that there are local rings $S$ admitting a pair of exact zero-divisors
$(x,x)$, but no local ring $T$ with regular element $\widetilde x$ such that $S=T/(\widetilde x^2)$
and $x=\widetilde x +(\widetilde x^2)$.  Therefore the notion of lifting modulo an exact zero-divisor
is a more general notion than lifting from modulo a regular element to modulo the square of the regular element.

\begin{example}\label{noembdim}
Let $k$ be field, $S=k[V,X,Y,Z]/I$ where $I$ is the ideal
\[
(V^2,Z^2,XY,VX+XZ,VY+YZ,VX+Y^2,VY-X^2),
 \]
and set $v=V+I$.  Then $(v,v)$ is a pair of exact zero-divisors.  Moreover, it is shown in \cite{AGP} that $S$ does not have an embedded deformation.  Therefore there is no local ring $T$ and non-zero-divisor $\widetilde V$ of $T$ such that $S\cong T/(\widetilde V^2)$.
\end{example}

\section*{Acknowledgments}
This paper is the result of a visit by the first and second authors to the third in October/November 2010. They thank the Department of Mathematics at the University of Texas at Arlington for their kind hospitality.

The authors thank the referee for his/her careful reading, diligence and helpful suggestions.


\begin{thebibliography}{AvBu}
\bibitem[ADS]{ADS} M.~Auslander, S.~Ding, and \O.~Solberg, \emph{Liftings and weak liftings of modules}, J. Algebra 156 (1993), 273--317.
\bibitem[Alp]{Alp} J. L. Alperin, \emph{Periodicity in groups}, Illinois J. Math., 21(4), 776-783, 1977.
\bibitem[AGP]{AGP} L. Avramov, V. Gasharov and I. Peeva, \emph{A periodic module of infinite virtual projective dimension}, J. Pure Appl. Algebra 62 (1989),\ 1-5.
\bibitem[Av1]{Avramov1}L.\ L.\ Avramov, \emph{Modules of finite virtual projective dimension}, Invent.\ Math.\ 96 (1989), no.\ 1, 71-101.
\bibitem[Av2]{Avramov2}L.\ L.\ Avramov, \emph{Homological asymptotics of modules over local rings}, Commutative algebra; Berkeley, 1987 (M.\ Hochster, C.\ Huneke, J.\ Sally, eds.), MSRI Publ.\ 15, Springer, New York 1989, pp.\ 33-62.
\bibitem[Av3]{Av3} L. L. Avramov, \emph{Infinite free resolutions, Six lectures on commutative algebra}, Bellaterra 1996, Progr. Math. 166, Birkh\"auser, Basel, (1998), 1-118.
\bibitem[AvBu]{AvBu}  L.\ L.\ Avramov and R.-O. Buchweitz, \emph{Support varieties and cohomology over complete intersections}, Invent. Math. 142 (2000), 285-318.
\bibitem[AH{\c{S}}]{AvramovHenriquesSega} L.\ L.\ Avramov, I.\ B.\ D.\ A.\ Henriques and L.\ M.\ {\c{S}}ega, \emph{Quasi-complete intersection homomorphisms}, to appear in Pure and Applied Mathematics Quarterly, posted at arXiv:1010.2143.
\bibitem[AV]{AV}L.L.\ Avramov, O.\ Veliche, \emph{Stable cohomology over local rings}, Adv.\ Math.\ 213 (2007), no.\ 1, 93-139.
\bibitem[Be]{Be} P. A. Bergh, \emph{On the vanishing of (co)homology over local rings}, J. Pure Appl. Algebra 212 (2008), no. 1, 262-270.
\bibitem[Dao1]{Dao} H. Dao, \emph{Some observations on local and projective hypersurfaces}, Math. Res. Lett., 15 (2008), 207-219.
\bibitem[Dao2]{Dao2} H. Dao,
\newblock Asymptotic behaviour of {T}or over complete intersections and applications,
\newblock preprint, (2006), posted at arxiv:07105818.
\bibitem[Eis]{Eisenbud} D.\ Eisenbud, \emph{Homological algebra on a complete intersection, with an application to group representations}, Trans.\ Amer.\ Math.\ Soc.\ 260 (1980), no.\ 1, 35-64.
\bibitem[Gu1]{Gulliksen1}T.\ H.\ Gulliksen, \emph{A change of ring theorem with applications to Poincar{\'e} series and intersection multiplicity}, Math.\ Scand.\ 34 (1974), 167-183.
\bibitem[Gu2]{Gulliksen2}T.\ H.\ Gulliksen, \emph{On the deviations of a local ring}, Math.\ Scand.\ 47 (1980), 5-20.
\bibitem[He{\c{S}}]{HenriquesSega}I.\ B.\ D.\ A.\ Henriques and L.\ M.\ {\c{S}}ega, \emph{Free resolutions over short Gorenstein local rings}, Math.\ Z.\ 267 (2011), 645-663.


\bibitem[Jo1]{Jo1} D. A. Jorgensen, \emph{Tor and torsion on a complete intersection}, J. Algebra 195 (1997), 526-537.
\bibitem[Jo2]{Jo2} D. A. Jorgensen, \emph{Complexity and Tor on a complete intersection}, J. Algebra 211 (1999), 578-598.
\bibitem[Rot]{Rot} J. Rotman, \emph{An introduction to homological algebra}, Universitext, Springer, New York, second edition, 2009.
\end{thebibliography}
\end{document}